\documentclass[12pt]{amsart}
\usepackage{epsfig,color}
\usepackage{blindtext}
\usepackage{graphicx}
\usepackage{float}
\usepackage[shortlabels]{enumitem}
\usepackage{url}
\usepackage{amssymb}
\usepackage{graphicx,import}
\usepackage{comment}
\usepackage{esint}
\usepackage{xcolor}
\usepackage{mathtools}
\usepackage{comment}
\usepackage{tikz}
\usepackage{thmtools}
\usepackage{thm-restate}
\usepackage{hyperref}
\usepackage{cleveref}

\declaretheorem[name=Lemma,numberwithin=section]{lemm}

\usepackage[margin = 1in] {geometry}

\usepackage{hyperref}
\usepackage{dsfont}

\newtheorem{theorem}{Theorem}[section]

\newtheorem{proposition}[theorem]{Proposition}
\newtheorem{lemma}[theorem]{Lemma}

\theoremstyle{definition}

\newtheorem{definition}[theorem]{Definition}
\newtheorem*{remark*}{Remark}

\newtheorem*{theorem*}{Theorem}

\numberwithin{equation}{section}

\newcommand{\R}{\mathbb{R}}

\newcommand{\N}{\mathbb{N}}
\newcommand{\Z}{\mathbb{Z}}

\newcommand{\length}{\mathrm{length}}

\DeclareMathOperator{\dmn}{domain}
\newcommand{\red}[1]{\textcolor{red}{#1}}

\newcommand{\purple}[1]{\textcolor{purple}{#1}}
\newcommand{\nl}{\newline}

\newcommand{\RR}{\mathbb{R}}
\renewcommand{\P}{\mathbb{P}}
\newcommand{\met}{\text{Met}}
\newcommand{\cS}{\mathcal{S}}

\title{The p-widths of $\mathbb{R} \mathbb{P}^2$}
\author{Jared Marx-Kuo}
\begin{document}

\begin{abstract}
We compute the p-widths, $\{\omega_p\}$, for the real projective plane with the standard metric.
\end{abstract}
\maketitle
\tableofcontents
\section{Introduction}
%\red{talk about p-widths generally. Yau's conjecture, etc. Talk about chodosh-mantoulidis. Talk about Ambrozio--Marques--Neves' work, and the rigidity of $\R \P^2$. Talk about how the p-widths for all $p$ has only been computed in 2 cases: on $S^2$, by Cho-Man, and then for Zoll metrics, by me. We now compute another example $\R \P^2$.} \newline 
%
In \cite{gromov2006dimension}, Gromov introduced the \textit{$p$-widths}, $\{\omega_p\}$, of a Riemannian manifold, $(M^{n+1}, g)$, as an analogue of the spectrum of the Laplacian, $\{\lambda_k\}$, the discrete collection of eigenvalues. Intuitively, one replaces the Rayleigh quotient in the definition of eigenvalues with the area functional in the definition of $\omega_p$ - see \cite{gromov2010singularities} \cite{gromov2002isoperimetry} \cite{guth2009minimax} \cite{fraser2020applications} for more background. The $p$-widths have proven to be extremely useful due to the Almgren--Pitts/Marques--Neves Morse theory program for the area functional (see \cite{almgren1965theory}\cite{pitts2014existence}\cite{marques2014min} \cite{marques2023morse} \cite{marques2015morse} \cite{marques2021morse} \cite{almgren1962homotopy} \cite{zhou2020multiplicity}). In \cite{liokumovich2018weyl}, Liokumovich--Marques--Neves demonstrated that $\{\omega_p\}$ satisfy a Weyl law, which later led to a resolution of Yau's conjecture \cite{yau1982problem}: Any closed three-dimensional manifold must contain an infinite number of immersed minimal surfaces. In some sense, the resolution was stronger than expected:
\begin{theorem}[\cite{chodosh2020minimal} \cite{song2018existence} \cite{marques2019equidistribution} \cite{marques2017existence} \cite{irie2018density} \cite{zhou2020multiplicity}] \label{yauThm}
On any closed manifold $(M^{n+1}, g)$ with $3 \leq n + 1 \leq 7$, there exist infinitely many embedded minimal hypersurfaces.
\end{theorem}
\noindent We refer the reader to the introduction of \cite{marx2024isospectral} for more information. In theorem \ref{yauThm}, the dimension upper bound is due to the existence of singular minimal hypersurfaces when $n+1 \geq 8$, while the lower bound is due to the fact that min-max on surfaces may only produce stationary geodesic networks in general (see \cite{pitts1973net}, \cite[Remark 1.1]{marques2015morse}). \newline 
\indent In a recent breakthrough \cite{chodosh2023p}, Chodosh--Mantoulidis used the Allen--Cahn equation to prove that the $p$-widths can always be achieved by a union of closed geodesics:
\begin{theorem}[\cite{chodosh2023p}, Thm 1.2] \label{ChoManPWidths}
	Let $(M^2, g)$ be a closed Riemannian manifold. For every $p \in \Z^+$, there exists a collection of primitive, closed geodesics $\{\sigma_{p,j}\}$ such that 
    \begin{equation} \label{widthsAchievedEq}
	\omega_p(M, g) = \sum_{j = 1}^{N(p)} m_j \cdot L_g(\sigma_{p,j}), \qquad \qquad m_j \in \Z^+
	\end{equation}
\end{theorem}
\noindent Here, ``primitive" means that the geodesic is connected and transversed with multiplicity one. The authors posed several questions about index and the number of intersections between geodesics in equation \eqref{widthsAchievedEq} in relation to $p$, and these were partially answered in \cite[Thm 1.3]{marx2024index}. Chodosh--Mantoulidis also showed:
\begin{theorem}[\cite{chodosh2023p}, Thm 1.4] \label{SphereWidths}
For $g_0$ the unit round metric on $S^2$
\[
\forall p \in \N, \qquad \omega_p(S^2, g_0) = 2 \pi \lfloor \sqrt{p} \rfloor
\]
\end{theorem}
\noindent Notably, theorem \ref{SphereWidths} was the first example of a manifold, $(M, g)$, for which $\{\omega_p(M, g)\}$ was explicitly known for all $p$. The authors also computed explicit sweepouts which saturate $\omega_p$. In \cite{marx2024isospectral}, the author showed that a connected component of Zoll metrics on $S^2$ have the same values of $\{\omega_p\}$:
\begin{theorem} \label{IsoSpectralThm}
Let $g$ be any Zoll metric on $S^2$ which lies in the connected component of $g_{0}$ in the space of Zoll metrics on $S^2$. Then $\omega_p(g) = \omega_p(g_{0}) = 2 \pi \lfloor \sqrt{p} \rfloor$ for all $p$.
\end{theorem}
\noindent Zoll metrics on $S^2$ are those for which all geodesics (with multiplicity one) are simple and of length $2 \pi$ (note this choice of length, as opposed to any constant $c$, agrees with \cite{guillemin1976radon}). Such metrics were first constructed by Zoll \cite{zoll1901ueber}, and we refer the reader to \cite{guillemin1976radon} \cite{funk1911flachen} \cite{gromoll1981metrics} \cite{besse2012manifolds} for a by no means complete list of background sources. \newline 
\indent In related work, Ambrozio--Marques--Neves \cite{ambrozio2021riemannian}constructed Zoll families of metrics on $S^n$, i.e. metrics which admit n-parameter families of $(n-1)$-hypersurfaces as the canonical family of equators. In \cite{ambrozio2024rigidity}, the same authors show that $(\R \P^2, g_{std})$ is p-width spectrally \textit{rigid}, i.e.
\begin{theorem}[Thm A, \cite{ambrozio2024rigidity}]
Let $(M^{n+1}, g)$ a compact riemannian manifold. If 
\[
\omega_p(M, g) = \omega_p(\R \P^2, g_{std})
\]
for all $p \in \mathbb{N}^+$, then $(M, g)$ is isometric to $(\R \P^2, g_{std})$.
\end{theorem}
\noindent The authors also prove a similar classification theorem of Zoll metrics on $S^3$ using a notion of spherical p-widths (\cite[Thm B]{ambrozio2024rigidity}). 
\subsection{Main Result + Ideas}
In this paper, we compute the full p-width spectrum for $\R \P^2$ with the standard metric, $g_{std}$, coming from quotienting the unit round metric on $S^2$ by the antipodal map.
\begin{theorem} \label{mainTheorem}
For $(\R \P^2, g_{std})$, we have 
\[
\omega_p = 2 \pi \cdot \Big\lfloor \frac{1}{4} \left(1 + \sqrt{1 + 8p} \right) \Big\rfloor \quad \forall p \in \mathbb{N}^+
\]
Moreover, $\omega_p$ is achieved by sweepouts coming from $\Z_2$ invariant polynomials on $S^2$. 
\end{theorem}
\noindent The proof is an adapation of the work done by Chodosh--Mantoulidis \cite{chodosh2023p} for the p-widths of $S^2$, with the following modifications:
\begin{enumerate}
    \item Showing that the min-max geodesics must occur with an even total multiplicity, i.e.
    \[
    \omega_p = \sum_{i = 1}^N a_i \cdot \gamma_i \implies \sum_{i = 1}^N a_i \equiv 0 \mod 2
    \]
    (see lemma \ref{lem:perturbedMetrics}, part 2)
    \item Considering sweepouts by level sets of polynomials of the form $f(x,y) + z g(x,y)$, where $f$ is even and $g$ is odd, so that the polynomial is $\Z_2$ invariant and descends to a map on $\R \P^2$ (see lemma \ref{lem:perturbedMetrics}, part 1).
\end{enumerate}
%
%Beyond this there are only minor modifications. \newline
%
In \S \ref{mmBackground}, we define the basic notions of Almgren--Pitts min-max theory as relevant to this paper. In \S \ref{lemmasSection}, we recall some essential propositions and lemmas from \cite{chodosh2023p}. We also construct the modified sweepouts to attain an upper bound for $\omega_p(\R \P^2)$. In \S \ref{ProofMain}, we prove the main theorem using the lemmas and a modified counting argument.

\subsection{Acknowledgements}
The author would like to thank Fernando Marques for suggesting the problem and Christos Mantoulidis and Otis Chodosh for many helpful discussions.
\section{Min--Max Background and Notation} \label{mmBackground}
We recall the essential notation from the min-max theory of Almgren--Pitts \cite{almgren1962homotopy} and Marques--Neves \cite{MN2020CETSurvey}. Let $M$ be an $(n+1)$-dimensional manifold, let $X \subseteq [0,1]^k$ denote a cubical subcomplex of the $k$-dimensional unit cube (cf. \cite[\S 2.2]{MN2020CETSurvey}), and let $\mathcal{Z}_n(M; \Z_2)$ denote the set of $n$-cycles mod $\Z_2$ - informally, a cycle is the boundary of a nice (Caccioppoli) open set. In \cite{almgren1962homotopy}, Almgren showed that $\mathcal{Z}_n(M; \Z_2)$ is weakly homotopic to $\mathbb{R} \mathbb{P}^{\infty}$. We let $\overline{\lambda} \in H^1(\mathcal{Z}_n(M; \Z_2), \Z_2) \cong \Z_2$ denote the generator of this group, so that $\overline{\lambda}^p$ generates $H^p(\mathcal{Z}_n(M; \Z_2), \Z_2) \cong \Z_2$. For $T \in \mathcal{Z}_n(M; \Z_2)$, there is a natural associated varifold to $T$ and we let $||T||$ denote the corresponding Radon measure, with $||T||(U)$ denoting the integral of the measure over $U$. We further let $\mathbf{M}(T) = ||T||(M)$. 
\begin{definition}[\cite{marques2017existence}, Defn 4.1] \label{defi:sweepout}
	A map $\Phi : X\to \mathcal{Z}_{n}(M;\Z_{2})$ is a \textbf{$p$-sweepout} if it is continuous (with the standard flat norm topology on $\mathcal{Z}_n(M; \Z_2)$) and $\Phi^{*}(\overline\lambda^{p}) \neq 0 \in H^p(X)$. 
\end{definition}

\begin{definition}[\cite{marques2017existence} \S 3.3] \label{defi:no.concentration.of.mass}
	A map $\Phi : X \to \mathcal{Z}_n(M; \Z_2)$ is said to have \textbf{no concentration of mass} if
	\[ \lim_{r \to 0} \sup \{ \Vert \Phi(x) \Vert(B_r(p)) : x \in X, p \in M \} = 0. \]
\end{definition}

\begin{definition}[\cite{gromov2002isoperimetry}, \cite{guth2009minimax}] \label{defi:p-width}
	We define $\mathcal{P}_{p} = \mathcal{P}_{p}(M)$ to be the set of all $p$-sweepouts, out of any cubical subcomplex $X$, with no concentration of mass. The \textit{$p$-width} of $(M,g)$ is
	\[
	\omega_{p}(M,g) := \inf_{\Phi \in \mathcal{P}_{p}} \sup\{\mathbf{M}(\Phi(x)) : x \in \mathrm{domain}(\Phi)\}.
	\]
\end{definition}
\noindent We note that one can work with refined sweepouts to compute $\omega_p$, i.e. let 
\[ 
\mathcal{P}^{\mathbf{F}}_{p,m} : = \{\Phi \in \mathcal{P}_{p} : \dmn(\Phi) \subset I^m \text{ and } \Phi \text{ is } \mathbf{F}\text{-continuous} \} 
\]
then:
\begin{lemm}[Lemma 2.6, \cite{chodosh2023p}] \label{lemm:P.p.m}
If $m = 2p+1$, then
\[
\omega_{p}(M,g) = \inf_{\Phi \in \mathcal{P}^{\mathbf{F}}_{p,m}} \sup\{\mathbf{M}(\Phi(x)) : x \in \dmn(\Phi)\}.
\]
\end{lemm}
\noindent To see the analogy between $\omega_p$ and $\lambda_p$, consider the following characterization of $\lambda_p(M)$ as
\begin{align*}
	\lambda_p(M) &:= \inf_{\substack{V \subseteq C^{\infty}(M) \\ \dim V = p+1}} \sup_{f \in V \backslash \{0\}} \frac{\int_M |\nabla f|^2}{\int_M f^2}
\end{align*}
Note the parallels between the min-max definitions of $\omega_p, \lambda_p$, with $\Phi^{*}(\overline\lambda^{p}) \neq 0$ being analogous to $\dim(V) = p+1$.

\section{Fundamental Lemmas} \label{lemmasSection}

\subsection{A good metric on $\R \P^2$}
Much of this section is analogous to \cite[\S 6]{chodosh2023p}: Consider the ellipsoids
\begin{equation} \label{ellipsoidEq}
E(a_1, a_2, a_3) := \{ (x_1, x_2, x_3) \in \RR^3 : a_1 x_1^2 + a_2 x_2^2 + a_3 x_3^2 = 1 \} \subset \RR^3 
\end{equation}
and the three geodesics
\[ \gamma_i(a_1, a_2, a_3) := E(a_1, a_2, a_3) \cap \{ x_i = 0 \}, \; i = 1, 2, 3 \]
on them. It is shown in \cite[Theorems IX 3.3, 4.1]{Morse:calculus.variations} that for every $\Lambda > 2\pi$, if $a_1 < a_2 < a_3$ are sufficiently close to $1$ (depending on $\Lambda$), then every closed connected immersed geodesic $\gamma \subset E(a_1, a_2, a_3)$ (coverings allowed) satisfies:
\begin{equation} \label{eq:geodesic.network.ellipsoid.geodesic.genericity}
	\gamma \text{ has no nontrivial normal Jacobi fields if } \length(\gamma) < 2\Lambda,
\end{equation} 
and
\begin{equation} \label{eq:geodesic.network.ellipsoid.geodesic.uniqueness}
	\gamma \text{ is an iterate of one of } \gamma_i(a_1, a_2, a_3), \; i = 1, 2, 3, \text{ if } \length(\gamma) < 2\Lambda.
\end{equation}
By quotienting by $(x_1, x_2, x_3) \sim (-x_1, -x_2, -x_3)$ the above ellipsoids descend to metrics on $\R \P^2$ such that equations \ref{eq:geodesic.network.ellipsoid.geodesic.genericity} \ref{eq:geodesic.network.ellipsoid.geodesic.uniqueness} hold by simply lifting the jacobi fields and geodesics to $S^2$. \newline 
\indent For $(a_1, a_2, a_3) \in \RR^3$, consider the vector $\vec{\ell}(a_1, a_2, a_3) \in \RR^3$ whose $i$-th component, $i=1, 2, 3$, is
\[ \big[ \vec{\ell}(a_1, a_2, a_3) \big]_i := \length(\gamma_i(a_1, a_2, a_3)). \]
It is easy to see that $(a_1, a_2, a_3) \mapsto \vec{\ell}(a_1, a_2, a_3)$ is smooth near $(1, 1, 1)$, and
\[ D\vec{\ell}(1, 1, 1) = \begin{pmatrix} 0 & \pi & \pi \\ \pi & 0 & \pi \\ \pi & \pi & 0 \end{pmatrix}. \]
Then, by the inverse function theorem there are smooth functions $\mu \mapsto a_i(\mu)$, $a_i(0) = 1$, $i = 1, 2, 3$, so that, for $\mu$ near $0$,
\[ \vec{\ell}(a_1(\mu), a_2(\mu), a_3(\mu)) = (2\pi, 2\pi + 2\mu, 2\pi+4\mu). \]
Thus, for sufficiently small $\mu$, 
\begin{equation} \label{eq:geodesic.network.ellipsoid.geodesic.lengths}
	\length(\gamma_i(a_1, a_2, a_3)) = 2\pi + (i-1) 2\mu.
\end{equation}
Moreover, these metrics again descend to $\R \P^2$, producing geodesics of length $\pi, \pi + \mu, \pi + 2\mu$.
%\red{I think this all works for $\R \P^2$ in the same way that it works for $S^2$}. \newline \newline
%
\indent We now recreate the following theorem, which is a direct analogue of \cite[Thm 6.1]{chodosh2023p}
\begin{theorem}\label{theo:good.metric.final}
	Let $\Lambda > 0$ and $U$ be any neighborhood, in the $C^\infty$ topology, of the standard metric $g_{std} \in \met(\R\P^2)$. There is a $\mu_0 = \mu_0(\Lambda, U) > 0$, so that for all $\mu \in (0, \mu_0)$, there exists $g_\mu \in U$ with all these properties:
	\begin{enumerate}
		\item There are simple closed geodesics $\gamma_1$, $\gamma_2$, $\gamma_3 \subset (\R \P^2, g_\mu)$ so that $\length_{g_\mu}(\gamma_i) = \pi + (i-1) \mu$.
		\item If a closed connected geodesic in $(\R\P^2, g_\mu)$ has $\length_{g_\mu} < \Lambda$, then it is an iterate of $\gamma_i$, $i = 1$, $2$, $3$.
		\item There are no not-everywhere-tangential stationary varifold Jacobi fields along any $S \in \cS^{\Lambda}(g_\mu)$.
	\end{enumerate}
	Moreover, $g_\mu \to g_{std}$ as $\mu \to 0$ in the $C^\infty$ topology.
\end{theorem}
\begin{proof}
	We choose $\mu_0$ small enough that $g^E_\mu \in U$ for all $\mu \in (0, \mu_0)$, where $g^E_\mu$ is the  induced metric of $E(a_1(\mu), a_2(\mu), a_3(\mu))/\{x \sim -x\}$.
	
	Fix any such $\mu$. By equation \eqref{eq:geodesic.network.ellipsoid.geodesic.genericity}, \cite[Corollary 5.35]{chodosh2023p} applies at $g^E_\mu$ with $2\Lambda$ in place of $\Lambda$ and with $S_i^0$, the multiplicity-one varifolds on $\gamma_i(a_1(\mu), a_2(\mu), a_3(\mu)) / \{x \sim - x \}$ for $i = 1$, $2$, $3$. As a consequence of the moduli space and generic regularity theory from Corollary 5.35 and Remark 5.34 of \cite{chodosh2023p} (both of which apply to any closed surface), there is a neighborhood $V \subset U$ of $g^E_\mu$ inside of which there is a dense set $D \subset V$ of metrics satisfying conclusions (1) and (3) of our theorem; denote the distinguished geodesics by $\gamma_j(g) \subseteq \R \P^2$, $j = 1$, $2$, $3$. 
	
	It remains to show that at least one of the metrics from $D$ satisfies conclusion (2) too and can be taken arbitrarily close to $g_\mu^E$. Suppose that were not the case. Take any sequence $\{ g_\mu^i \}_{i=1}^{\infty} \subset D$ with $g_\mu^i \to g^E_\mu$ as $i \to \infty$. Let $\gamma_\mu^i$ be a closed connected geodesic in $(\R \P^2, g_\mu^i)$ with $\length_{g_\mu^i}(\gamma_\mu^i) < \Lambda$ that is not an iterate of any of $\gamma_j(g_\mu^i)$, $j = 1, 2, 3$. Pass to $i \to \infty$ along a subsequence (not relabeled) so that $\gamma^i_\mu \to \gamma^E_\mu$, a geodesic in $(\R \P^2, g_\mu^E)$ with $\length_{g^E_\mu}(\gamma^E_\mu) \leq \Lambda < 2 \Lambda$. By \eqref{eq:geodesic.network.ellipsoid.geodesic.uniqueness}, $\gamma^E_\mu$ is an iterate of a $\gamma_j(g^E_\mu)$, $j = 1, 2, 3$. So, $\gamma^i_\mu$ is $o(1)$-close (as $i \to \infty$) to being an iterate of $\gamma_j(g^i_\mu)$. However, recall that the three $\gamma_j(g_\mu)$ and their iterates with length $< 2 \Lambda$ are isolated in $(\R \P^2, g_\mu)$ by equation \eqref{eq:geodesic.network.ellipsoid.geodesic.uniqueness}. Therefore, the geodesics $\gamma_j(g^i_\mu)$ and their iterates with length $\leq \Lambda$ are isolated in $(\R \P^2, g^i_\mu)$, contradicting the existence of $\gamma^i_\mu$ when $i$ is large.
    %(\red{This follows by constructing a Jacobi field right? If $\gamma_{\mu}^i$ was close to $\gamma_j(g_{\mu}^i)$})
    This completes the proof.
\end{proof}
\subsection{Upper Bound}
We recreate the analogous upper bound on the p-widths, by using a restricted class of sweepouts on $\R \P^2$. First, let $D(d) = (d+1)(2d+1)$, and let $f: \Z^+ \to \Z^+$ be such that 
\[
f(p) = d + 1 \iff p \in \{D(d), \dots, D(d+1)-1\}
\]
\begin{lemma} \label{lem:perturbedMetrics}
For $p \in \N^*$, there eixsts $\mu_1 > 0$ and an open neighborhood $U$ of the unit round metric on $\R \P^2$ depending on $p$ such that for all $\mu \in (0, \mu_1)$, $g_{\mu} \in U$ and 
\begin{enumerate}
    \item \label{enum:upperBound} $\omega_p(\R\P^2, g_{\mu}) \leq 2 \pi f(p) + 1$
    \item \label{enum:containment} For any $\mathcal{F}$-homotopy class $\Pi \subseteq \mathcal{P}^{\mathbf{F}}_{p,m}$ (i.e. $p$-sweepouts coming from domains of dimension at most $m$)
    \begin{align*}
    \mathbf{L}_{AP}(\Pi, g_{\mu}) &\in  \Big( \{ n_1 \cdot \pi + n_2 \cdot (\pi + \mu) + n_3 \cdot (\pi + 2 \mu) \; \\
    & \qquad : \; (n_1,n_2,n_3) \in \N^3 \} \backslash \{0\}, \; n_1 + n_2 + n_3 \equiv 0 \mod 2 \Big) \\
    & \qquad \cup [2\pi f(p) + 2, \infty)
    \end{align*}
    %
    %\red{Probably want to refine this argument so that the coefficients are $2\pi$ somehow...}
\end{enumerate}
\end{lemma}
\begin{proof}
To verify \ref{enum:upperBound}, let 
\[
A_{2d} = \{f(x,y) + z g(x,y) \; | \; f \in \R_{e,2d}[x,y], \quad g \in \R_{o, 2d-1}[x,y]\}
\]
Here $\R_{e,2d}[x,y]$ denotes the set of all even polynomials in $x$ and $y$ of degree $\leq 2d$ and $\R_{o, 2d-1}$ the analogous space of odd polynomials of degree $\leq 2d-1$. Note that 
\begin{align*}
\dim(A_{2d}) &= \left( \sum_{i = 0}^{d} (2i+1)\right)  + \left(\sum_{i=1}^d 2i\right) \\
&= \left[ 2\frac{d(d+1)}{2} + (d + 1) \right]  + \left[ 2 \frac{d(d+1)}{2} \right] \\
&= (d+1)^2 + d(d+1) \\
&= (d+1)(2d + 1) = D(d)
\end{align*}
Let $\{p_i^{e,2d}\}$ be an enumeration of even polynomials of order $\leq 2d$ and $\{p_j^{o,2d-1}\}$ the analgoous for odd polynomials of order $\leq 2d-1$. For $d \geq 1$, we define a $D(d)-1$-sweepout of $(\R \P^2, g_{std})$ as follows
\begin{align*}
F_d&: \R \P^{D(d) - 1} \to \mathcal{Z}_1( \R \P^2, \Z_2) \\
F_d([a_1 : \cdots : a_{D(d)}]) &= \Big\{ \sum_{i = 1}^{(d+1)^2} a_{i} p^{e,2d}_{i}(x,y) + z\sum_{j = 1}^{d(d+1)} a_{(d+1)^2 + j} p^{o,2d-1}_{j}(x,y) = 0\Big\}
\end{align*}
Note the $(x,y,z) \sim (-x,-y,-z)$ invariance of $F_d([a_1:\cdots:a_{D(d)}])$ so that it produces a well defined element of $\mathcal{Z}_1(\R \P^2, \Z_2)$. We see that $F_d$ is $D(d)-1$-sweepout by checking that it is a $1$-sweepout when restricting to the copy of $\R \P^1$ realized by $0 = a_3 = \dots a_{D(d)}$ as then
\[
F_d([a_1: a_2: 0 : \cdots : 0]) = \{a_1 + a_2 x^2 = 0\}
\]
where by convention, we assume that $p_1^{e,2d} = 1$ and $p_2^{e,2d} = x^2$. \newline 
\indent We now compute an upper bound on $\mathbf{M}(F)$ using the crofton formula, i.e.
\begin{align*}
\mathbf{M}(F_d([\vec{a}])) &= \frac{1}{2} \mathbf{M}(\tilde{F}_d([\vec{a}])) = \frac{1}{2} \frac{1}{4} \int_{\xi \in S^2} \sharp (\tilde{F}([\vec{a}]) \cap \xi^{\perp})
\end{align*}
where $\tilde{F}_d([\vec{a}]) \in \mathcal{Z}_1( S^2, \Z_2)$ is the preimage of $F_d([\vec{a}])$ under the antipodal map (recall we are working with the unit round metric still), and $\xi^{\perp}$ denotes the great circle which lies in the plane in $\R^3$ normal to $\xi$. Note that every such plane is given by $\{\alpha x + \beta y + \gamma z = 0\}$ for some $(\alpha, \beta, \gamma) \in \mathbb{R}^3$. Then 
\begin{align*}
\tilde{F}_d([\vec{a}]) \cap \xi^{\perp} &= \{ \sum_{i = 1}^{(d+1)^2} a_{i} p^{e,2d}_{i}(x,y) + z\sum_{j = 1}^{d(d+1)} a_{(d+1)^2 + j} p^{o,2d-1}_{j}(x,y) = 0\} \cap \{x^2 + y^2 = 1\} \\
& \qquad \quad \cap \{\alpha x + \beta y + \gamma z = 0\}
\end{align*}
By Bezout's inequality, we have that 
\[
\sharp \left(\tilde{F}_d([\vec{a}]) \cap \xi^{\perp}\right) \leq 2d \cdot 2 \cdot 1 = 4d
\]
%
\begin{comment}
\red{I think Aiex claims $4d$, for at least the case of $d = 1$. Some argument is wrong in my counting the number of intersection points} (\purple{Update: the degree argument is wrong, see Bezout's inequality in algebraic geometry which says that the number of solutions is at most the product of the degrees of the polynomials in the system of algebraic equations}
having substituted $y^2 = 1 - x^2$ and $z = 0$. And hence, there are at most $2d$ points of intersection, i.e. 
\[
\sharp (\tilde{F}([\vec{a}]) \cap \xi^{\perp}) \leq 2d
\]
By symmetry, this must be true for all $\xi \in S^2$, and so we conclude
%
\begin{align*}
\mathbf{M}(F_d([\vec{a}])) &\leq \frac{1}{2} \frac{1}{4} \int_{\xi \in S^2} 2d \\
& \leq \frac{1}{8} \cdot 4 \pi \cdot 2d \\
&= \pi \cdot d
\end{align*}
%
%\red{Off by a factor of $2$ somehow. This says that $\omega_1, \omega_2, \dots, \omega_5 \leq \pi$, which is too small of a bound. Need it to be $2\pi$} \newline 
%Note that 
%\[
%dd
%\]
%See here: https://math.stackexchange.com/questions/1529268/simple-proof-of-the-cauchy-crofton-formula-on-the-sphere
%
\end{comment}
%
Thus, we conclude 
\begin{align*}
\mathbf{M}(F_d([\vec{a}])) &\leq \frac{1}{2} \frac{1}{4} \int_{\xi \in S^2} 4d \\
& \leq \frac{1}{8} \cdot 4 \pi \cdot 4d \\
&= 2 \pi \cdot d
\end{align*}
%
%\red{Should be $2 \pi d$} \nl \nl 
%
Now notice that if $f(p) = d$ then $F_{d}$ is a p-sweepout, and so we conclude 
\[
\omega_p(\R \P^2, g_{std}) \leq 2 \pi f(p)
\]
and by continuity of the p-widths (see \cite[Lemma 2.1]{irie2018density}), we have that if $U$ is a sufficiently small neighborhood of $(\R \P^2, g_{std})$, then for any $g \in U$,
%for $\mu$ sufficiently small that 
\[
\omega_p(\R \P^2, g) \leq 2 \pi f(p) + 1
\]
Similarly, given such a neighborhood $U$, we can apply theorem \ref{theo:good.metric.final} to find $g_{\mu} \in U$ for all $\mu = \mu_0(2 \pi p + 2, U)$ and the above bound will also hold for $g = g_{\mu}$. \nl 
\indent To verify part \ref{enum:containment}, fix the metric $g_{\mu}$ as above and choose an $\mathcal{F}$ homotopy class with value $L_{AP}(\Pi) < 2 \pi f(p) + 2$. By theorem \ref{ChoManPWidths}, we have that 
\[
L_{AP}(\Pi) = \sum_{j = 1}^N a_j \cdot \ell_{g_{\mu}}(\sigma_j) \qquad a_j \in \Z^{>0}
\]
where the $\sigma_j$ are distinct, simple, and closed geodesics. Since $\ell(\sigma_j) < 2 \pi f(p) + 2$, by theorem \ref{theo:good.metric.final}, we conclude that $\sigma_j = \gamma_i$ for some $i \in \{1,2,3\}$ and hence $\ell(\sigma_j) \in \{\pi, \pi + \mu, \pi + 2 \mu\}$. This gives 
\[
L_{AP}(\Pi) = (n_1 + n_2 + n_3) \pi + \mu (n_2 + 2n_3) \qquad n_i \in \Z^{>0}
\]
To see that $n_1 + n_2 + n_3 \equiv 0 \mod 2$, we first rewrite the critical varifold as 
\[
L_{AP}(\Pi) = \sum_{j = 1}^3 b_j \cdot \ell_{g_{\mu}}(\sigma_j) 
\]
And it suffices to show that 
\[
\sum_{j = 1}^3 b_j \equiv 0 \mod 2
\]
To see this, note that the corresponding cycle mod $\Z_2$ must be trivial, i.e.
\[
\sum_{j = 1}^3 b_j \cdot [[\sigma_j]] \equiv 0 \in H_1(\R \P^2, \Z_2)
\]
This follows since $\sigma = \sum_{j=1}^3 b_j \cdot \sigma_j \in \mathbf{C}_{AP}(\mathcal{P}_{p})$, where $\mathbf{C}_{AP}(\mathcal{P}_{p})$ denotes the union of the Almgren--Pitts critical sets of all homotopy classes $\Pi \subseteq \mathcal{P}_{p}$. By definition, there exists, $\Pi \subseteq \mathcal{P}_p$ and $\Phi_i \in \Pi$ such that $\Phi_i: X_i \to \mathcal{Z}_1( \R \P^2, \Z_2)$ and $x_i \in X_i$ such that 
\[
\lim_{i \to \infty} \mathbf{F}(\Phi_i(x_i), \sigma) = 0
\]
Moreover, each $\Phi_i(x_i)$ is trivial homologically by definition of being an element of $\mathcal{Z}_1$. Convergence in the $\mathbf{F}$ norm preserves homology, so we conclude that $\sigma = 0 \in H_1(\R \P^2, \Z_2)$. \nl 
%
\begin{comment}
(\purple{This is homologically trivial because its the limit of images of a sweepout, $\Psi_i(t_i)$. But this limit converges in the flat norm, so homologically trivial is preserved because $\Psi_i(t_i)$ always homologically trivial}

(\red{Is this exactly true? I know for every element of the homotopy class, the map produces trivial elements of homology, but for the critical set, its not guaranteed. I guess limits of trivial elements of homology are trivial, though I'd have to think about it...}) \nl 
\end{comment}
%
\indent However, note that each individual $\sigma_j \not = 0 \in H_1(\R \P^2, \Z_2)$ - one can see this by considering the diffeomorphism $(\R \P^2, g_{\mu}) \to (\R \P^2, g_{std})$ which maps each $\sigma_j$ to one of the corresponding $3$ axial geodesics in $(\R \P^2, g_{std})$ (see the defining equation \eqref{ellipsoidEq} and then quotient), corresponding to $\gamma_j = \{x_j = 0\} \cap (\R \P^2, g_{std})$ for $j = 1,2,3$. However, each of these geodesics are homologically non-trivial.
%
\begin{comment}
\indent \red{Oh wait this is not true, I mean $H_1(S^2) = 0$ anyway, so all geodesics are trivial. Makes sense since they are all boundaries. It should just be that each $\sigma_i \not = 0 \in H_1(\R \P^2, \Z_2)$} We can lift $\sigma$ to $S^2$, and note that the lift is trivial in $H_1(S^2, \Z_2)$, i.e.
\[
\sum_{j = 1}^3 b_j \cdot [[\tilde{\sigma}_j]] \equiv 0 \in H_1(S^2, \Z_2)
\]
where $\tilde{\sigma}_j$ are the lifts of $\sigma_j$ to $S^2$, and by the proof of theorem \ref{theo:good.metric.final} (see also \cite[Thm 6.1]{chodosh2023p}), for $\mu_0$ sufficiently small, each $\tilde{\sigma}_j$ is simple, closed, and homologically non-trivial. Thus
\[
\sum_{j = 1}^3 b_j \cdot [[\tilde{\sigma}_j]] \equiv 0 \in H_1(S^2, \Z_2) \iff \sum_{j=1}^3 b_j \equiv 0
\]
\end{comment}
%
Thus, we conclude that 
\[
\sum_{j = 1}^3 b_j \cdot [[\sigma_j]] \equiv 0 \in H_1(\R \P^2, \Z_2) \iff \sum_{j=1}^3 b_j \equiv 0
\]
%Note that in the above, WLOG we have $a_j \in \{0,1\}$, and each
%, i.e. it suffices to consider the case when our cycle of interest
\end{proof}
\noindent Having proved lemma \ref{lem:perturbedMetrics} and theorem \ref{theo:good.metric.final}, we note that the below proposition from \cite{chodosh2023p} holds for $\R \P^2$ instead of $S^2$, with only minor notational changes.
\begin{proposition}[Prop 6.11, \cite{chodosh2023p}] \label{widthInequality}
Fix $p \in \mathbb{N}$ and $\mu_1(p) > 0$ as in lemma \ref{lem:perturbedMetrics}. Then 
\[
\omega_p(\R \P^2, g_{\mu}) < \omega_{p+1}(\R \P^2, g_{\mu})
\]
for every $\mu \in (0, \mu_1)$.
\end{proposition}
\section{Proof of Main Theorem \ref{mainTheorem}} \label{ProofMain}
\noindent We now compute all of the widths of $(\R \P^2, g_{std})$
\begin{theorem*}
For $(\R \P^2, g_{std})$, we have for any $p \in \mathbb{N}^+$,
\[
\omega_p = 2 \pi f(p) = 2 \pi \Big\lfloor \frac{1}{4} \left(1 + \sqrt{1 + 8p} \right) \Big\rfloor
\]
\end{theorem*}
\begin{proof}
%\red{Actually I think this will be different}
As in the proof of theorem \ref{SphereWidths} from \cite{chodosh2023p}, we will show that 
\[
\omega_{D(d)} = \omega_{D(d) + 1} = \cdots = \omega_{D(d+1) - 1} = 2 \pi (d+1)
\]
for all $d$. Let $\mu_1 = \mu_1(p = D(d+1) - 1) > 0$ as in lemma \ref{lem:perturbedMetrics} and suppose that $0 < \mu < \min(\mu_1, 1/2(d+1))$. \nl 
\indent As in \cite[Claim 7.1]{chodosh2023p}, we want to prove:
\begin{align*}
\{\omega_p(\R \P^2, g_{\mu}) \; : \; p = 1, \dots, D(d+1) - 1\} & = \Big( \{ n_1 \cdot \pi + n_2 \cdot (\pi + \mu) + n_3 \cdot (\pi + 2 \mu) \; \\
    &  \quad : \; (n_1,n_2,n_3) \in \N^3 \} \backslash \{0\},  \\
    & \quad n_1 + n_2 + n_3 \equiv 0 \mod 2 \} \Big) \cap (0, 2 \pi (d+1) + 1]
\end{align*}
Denote the set on the left as $L$ and on the right as $R$. $R \subseteq L$ follows directly from lemma \ref{lem:perturbedMetrics}. For the other containment, we note that by proposition \ref{widthInequality}:
\begin{align*}
\sharp \{\omega_p(\R \P^2, g_{\mu}) \; : \; p = 1, \dots, D(d+1) - 1\} & = D(d+1) - 1 = (d+2)(2d + 3) - 1 \\
&= 2d^2 + 7d + 5 = (2d + 5)(d+1)
\end{align*}
To compute $|R|$, we argue as in \cite[Claim 7.1]{chodosh2023p} and write
\begin{align*}
n_1 \cdot \pi + n_2 \cdot (\pi + \mu) + n_3 \cdot (\pi + 2 \mu) &= \pi \cdot (n_1 + n_2 + n_3) + \mu (n_2 + 2 n_3)
\end{align*}
And count the number of values of $\{\mu (n_2 + 2n_3) \; | \; n_1 + n_2 + n_3 = 2j\}$ where $0 \leq 2j \leq 2(d+1)$. However
\[
\{\mu (n_2 + 2n_3) \; | \; n_1 + n_2 + n_3 = 2j\} = \{0, \mu, 2 \mu, \dots, 4 j \mu \}
\]
and so 
\begin{align*}
|R| &= \sum_{j = 1}^{d+1} \sharp \{ \mu (n_2 + 2n_3) \; | \; n_1 + n_2 + n_3 = 2j\} \\
&= \sum_{j = 1}^{d+1} 4j + 1 \\
&= 4 \cdot \frac{(d+1)(d+2)}{2} + d+1 \\
&= (d+1) \cdot (2d + 5)
\end{align*}
And so necessarily we have $L = R$. \newline 
\indent Now by induction and the fact that the widths are increasing (proposition \ref{widthInequality}), we see that 
\begin{align*}
\forall p \in \{D(d), D(d) + 1, \dots, D(d+1) - 1\} \\
2\pi \cdot (d+1) \leq \omega_p(\R \P^2, g_{\mu}) \leq (2\pi + 4\mu ) \cdot (d+1)
\end{align*}
using continuity of the p-widths in the metrics and nothing $g_{\mu} \xrightarrow{\mu \to 0} g_{std}$, we conclude that 
\begin{equation} \label{widthEquation}
p \in \{D(d), D(d) + 1, \dots, D(d+1) - 1\} \implies \omega_p(\R \P^2, g_{std}) = 2 \pi \cdot (d+1) = 2\pi \cdot f(p)
\end{equation}
These values are achieved precisely by the sweepouts defined in \ref{lem:perturbedMetrics} by the upper bound computed there and equation \eqref{widthEquation}. The fact that $f(p) = \lfloor \frac{1}{4} \left(1 + \sqrt{1 + 8p} \right)\rfloor$ follows from algebraic manipulation.
%the size of the set on the left hand side is exactly $D(d+1) = $
%(\red{Okay here we need something new. One containment \textit{does not} follow from the lemma since geodesics on $\R \P^2$ are length $\pi$, not $2\pi$. Probably need to go back to the lemma and show that we must have even multiplicity. Makes sense, min-max geodesics in unoriented manifolds probably occur with even multiplicity. Until we have this, the counting argument won't hold})
\end{proof}

\bibliography{main}{}
\bibliographystyle{amsalpha}
\end{document}